\documentclass[3p,final]{elsarticle}

\usepackage{hyperref}
\usepackage{amssymb}
\usepackage{amsmath}
 \usepackage{amsthm}
 \newtheorem{definition}{Definition}[section]
\newtheorem{theorem}[definition]{Theorem}

\newtheorem{lemma}[definition]{Lemma}
\newtheorem{corollary}[definition]{Corollary}

\usepackage{subfloat}
\usepackage{hyperref}
\newtheorem{Observation}[definition]{Observation}

\usepackage{tikz}
\usepackage{latexsym}  
\usepackage{pgf}
\usetikzlibrary{positioning}
\usetikzlibrary{decorations.pathmorphing,arrows}
\usepackage{natbib}

\journal{Discrete Applied Mathematics}

\begin{document}

\begin{frontmatter}
\title{On chordal and perfect plane near-triangulations}
%\title{A local characterization for perfect plane near-triangulations\tnoteref{mytitlenote}}

%\author{Sameera Muhamed Salam \fnref{myfootnote}}
%\author{Sameera Muhamed Salam}
%\address{Department of Computer Science, NIT Calicut}
%\author{Jasine Babu}
%\address{Department of Computer Science, IIT Palakad}
%\author{K Murali Krishnan}
%\address{Department of Computer Science, NIT Calicut}

\author{Sameera M. Salam}
\ead{shemi.nazir@gmail.com}
\author{Nandini J. Warrier}
\ead{nandini.wj@gmail.com}
\author{Daphna Chacko}
\ead{daphna.chacko@gmail.com}
%\author[mysecondaryaddress]{Jasine Babu\corref{mycorrespondingauthor}}
%\author[mysecondaryaddress]{Jasine Babu}
%\ead{jasine@iitpkd.ac.in}
\author{K. Murali Krishnan}
\ead{kmurali@nitc.ac.in}
\author{Sudeep K. S.}
\ead{sudeep@nitc.ac.in}
\address{Department of Computer Science and Engineering, National Institute of Technology Calicut, Kerala, India 673601}
%\address[mysecondaryaddress]{Department of Computer Science and Engineering, Indian Institute of Technology, Palakkad, Kerala, India 678557}

\begin{abstract}
A plane near-triangulation $G$ can be decomposed into a collection of induced subgraphs,
described here as the W-components of $G$, such that $G$ is perfect 
(respectively, chordal) 
if and only if each of its W-components is perfect (respectively, chordal). 
Each W-component is a $2$-connected plane near-triangulation, free of edge
separators and separating triangles.  Graphs satisfying these conditions will be  
called W-near-triangulations.  A linear time decomposition of $G$ into its 
W-components is achievable using known techniques from the literature. 

W-near-triangulations have the property that the open neighbourhood of every internal
vertex induces a cycle.  It follows that a W-near-triangulation $H$ of at 
least five vertices is non-chordal if and only if it contains an internal vertex. 
This yields a local structural characterization that a 
plane near-triangulation $G$ is chordal if and only if it does not contain 
an induced wheel of at least five vertices.  

For W-near-triangulations that are free of induced wheels of five vertices,
we derive a similar local criteria, that depends only on the neighbourhoods of individual
vertices and faces, for checking perfectness.  We show that a W-near-triangulation $H$ that is free
of any induced wheel of five vertices is perfect if and only if there exists neither an internal vertex 
$x$, nor a face $f$ such that, the neighbours of $x$ or $f$ induces an odd hole.   
The above characterization leads to a linear time algorithm for determining 
perfectness of this class of graphs.   
\end{abstract}

\end{frontmatter}

%\linenumbers
\section{Introduction}
A plane embedding of a (planar) graph is called a plane near-triangulation 
if the boundary of every face, except possibly the outer face, 
is a cycle of length three. 
We try to derive local characterizations for checking whether a
plane near-triangulation is chordal or perfect. Here, a \textit{local characterization}
refers to a condition that can be checked by inspecting the neighbourhood 
of individual vertices, edges or faces of the graph.  
A graph is chordal if and only if it
is free of induced cycles of length exceeding three
\cite{west1996introduction}.  A graph is perfect if and only if it is
free of induced odd cycles of length exceeding three (or odd holes)
\cite{tucker1973strong}.  

Investigation of the structural properties of plane triangulations and some 
of their subfamilies like Apollonian networks have been elaborately undertaken
in the literature \cite{laskar2011maximal, biedl2004tight, Kumar:1989:NCS:646827.707381,cahit2003characterization}, 
owing to their rich and interesting geometric structure.  Here we investigate local 
structural characterizations for chordal and perfect plane near-triangulations.    

A plane near-triangulation $G$ 
can be decomposed in linear time, into a set of induced component subgraphs, 
which we call the {\em W-components} of $G$ (see Section~\ref{wdecomposition}).    
Each W-component $H$ of $G$ is essentially a $2$-connected 
plane near-triangulation that is free of edge separators and separating
triangles. Graphs satisfying these conditions are referred to as  
{\em W-near-triangulations}.  The neighbourhood of every 
internal vertex of any W-near-triangulation induces a wheel.

The problem of determining whether a plane
near-triangulation $G$ is chordal (respectively, perfect) can be reduced to 
the problem of checking whether each of its W-components 
is chordal (respectively, perfect).  In Section~\ref{wdecomposition} we describe 
a linear time procedure to extract the W-components of $G$,
by adapting a method known in the literature \cite{kant1997more} for identifying
4-connected blocks in a plane triangulation.  

% W-components of $G$ can be extracted
% by first decomposing $G$ into $2$ connected blocks,  
% triangulating each $2$ connected block by adding an extra vertex in the exterior
% face if needed, identifying the four-connected-blocks in each triangulation 
% using methods known in the literature [REF TO KANT], and finally removing the 
% newly added vertex from each $4$-connected block (whenever it is present).  
% Each of the above steps requires only linear time.  

In Section~\ref{sec2} it is shown that a W-near-triangulation $G$ that is not $K_4$ is chordal 
if and only if it does not contain an internal vertex.  
Consequently, we derive a local structural characterization 
that a  plane near-triangulation $G$ is chordal if and only if it does not contain 
an induced wheel of at least five vertices.  

In Section~\ref{sec4}, we show that perfect W-near-triangulations that do not 
contain any induced wheel of five vertices admit a simple local 
characterization.  It is shown that a W-near-triangulation $H$ that does not 
contain any induced wheel of five vertices is perfect 
if and only if there exists neither an internal vertex 
$x$, nor a face $f$ in $H$ such that, the neighbours of $x$ or $f$ induces 
an odd hole.   This local structural characterization results in a linear 
time algorithm for determining whether a W-near-triangulation, that is 
free of any induced wheel of five vertices, is perfect.  
No sub-quadratic time algorithm appears to be known for 
recognizing perfect plane near-triangulations or perfect plane triangulations.  
%%%%%%%%%%%%%%%%%%%%%%%%%%%%%%%%%%%%%%%%%%%%%%%%%%%%%%%%%%%%%%%%%%%%%%%%%%%%%%%%%%%%%%%%%%%%%%%%%%%%%%%%%%%%%%%%%%%%%%
\section{Preliminaries}
Given a plane near-triangulation $G$, we call the vertices on the boundary of the external
face of $G$ as the {\em external vertices} of $G$, denoted by $Ext(G)$ and the remaining vertices
as the {\em internal vertices} of $G$, denoted by $Int(G)$.  
The notation $C_n$ will be used to denote a cycle of $n$ vertices.  If $S\subseteq V(G)$, then $N_{G}(S)$ (respectively, $N_{G}[S]$) denotes 
the open (respectively, closed) neighbourhood 
of the set $S$.   In the case when $S=\{u\}$ for
a vertex $u\in V(G)$, we write $N_{G}(u)$ (respectively, $N_{G}[u])$ 
for the open (respectively, closed) neighbourhood of $u$. 
The suffix will be dropped when the underlying graph $G$ is clear from the context.
% Let $G = (V,E)$ be a simple undirected graph. A \textit{drawing} of a graph maps 
% each vertex $u \in V$ to a point $\varepsilon(u)$ in $\mathbb{R}^2$
% and each edge $uv \in E$ to a path with endpoints $\varepsilon(u)$ and $\varepsilon(v)$. 
% The drawing is a plane embedding if the points are distinct, the paths are simple and do not cross each other, 
% and the incidences are limited to the endpoints. The well known Jordan Curve Theorem \cite{hales2007jordan} states that
% if $J$ is a simple closed curve in $\mathbb{R}^2$, then $\mathbb{R}^2-J$
% has two components, the interior of $J$ (denoted by $Int(J))$ and the 
% exterior of $J$ (denoted by $Ext(J)$), with $J$ as the boundary of each. 
% Given a plane near-triangulation $G$, $Int(G)$ ($Ext(G)$) denotes the set of vertices in the interior (and respectively, exterior) of the closed curve defined by the exterior boundary of the triangulation. 
% 
% %\subsection*{Definitions}
 \begin{definition}[Wheel]\label{Def1}
A wheel on $n$ ($n\ge4$) vertices, $W_n$, is the graph obtained by adding a new vertex $v$ to a cycle $C_{n-1}$ and making 
it adjacent to all vertices in $C_{n-1}$. The cycle $C_{n-1}$ is called the rim of the wheel, 
the vertex $v$ is called the centre of the wheel and the added edges joining $v$ and  vertices in $C_{n-1}$ are called spokes of the wheel.  
\end{definition}
A wheel $W_n$ is called an {\em even wheel} (respectively, {\em odd wheel}) 
if $n$ is even (respectively, odd).  Note that the rim of an even wheel contains
an odd number of vertices and the rim of an odd wheel has an even number of vertices.  
Any induced cycle of length at least four in a graph is called a {\em hole}. 
A hole with odd number of vertices is known as an {\em odd hole}.
A {\em separator} in a connected graph is a set of vertices, 
the removal of which disconnects the graph. 
A clique in a graph is a set of pairwise adjacent vertices. 
A {\em clique separator} is a separator which is a clique.
A clique separator of size two (respectively, three) is 
called an {\em edge separator} (respectively, a {\em separating triangle}).

\begin{definition}[W-near-triangulation]\label{Def3}
A plane near-triangulation $G$ is called a W-near-triangulation 
if $G$ is two connected and, either $G$ is isomorphic to $K_4$ or 
$G$ contains neither a separating triangle, nor an edge separator.  
A W-near-triangulation $G$ is called an even W-near-triangulation 
if the degree of every vertex in $Int(G)$ is even. 
\end{definition}
Note that a W-near-triangulation need not be $4$-connected 
(for example, a wheel on five vertices is a W-near-triangulation,
but contains a $3$-separator).

In the next section, we show that the study of chordality 
(respectively, perfectness) of  
plane-near-triangulations reduces to the study of chordality
(respectively, perfectness) of W-near-triangulations.   

\section{W-decomposition}\label{wdecomposition}
In this section we describe a method to decompose any plane near-triangulation $G$ into a collection of induced subgraphs, 
$G_1,G_2,\ldots, G_k$ (for some $k\geq 1$) in linear time, where each 
$G_i$, $i\in \{1,2\ldots, k\}$ is a W-near-triangulation and
$G$ is chordal (respectively, perfect) if and only if all of $G_1,G_2,\ldots G_k$
is chordal (respectively, perfect).
%We describe how any plane near-triangulation $G$ can be decomposed 
%in linear time into a collection of induced subgraphs, 
%$G_1,G_2,\ldots, G_k$ (for some $k\geq 1$), where each 
%$G_i$, $i\in \{1,2\ldots, k\}$ is a W-near-triangulation and
%$G$ is chordal (respectively, perfect) if and only if all of $G_1,G_2,\ldots G_k$
%is chordal (respectively, perfect). 
The method described here is a combination of known techniques for handling plane 
triangulations, drawn from various sources.   We sketch the details briefly here for
the sake of completeness.  

Let $G$ be a plane near-triangulation. $G$ is chordal (respectively, perfect)
if and only if each of its $2$-connected blocks is chordal (respectively, perfect).
Since we can identify the $2$-connected blocks of $G$ in linear time, we 
assume hereafter that $G$ is $2$-connected.  

Let $uv$ be an edge separator in $G$. Let $H_1$ and $H_2$ be the two components of 
$G\setminus \{u,v\}$.  It is easy to see that $G$ is chordal (respectively, perfect) 
if and only if the subgraphs $G_1$ and $G_2$ induced by 
$V(H_1)\cup\{u,v\}$ and $V(H_2)\cup\{u,v\}$ are chordal (respectively, perfect).  
Similarly, let $uvw$ is a separating triangle in $G$ and, let  
$H_1$ and $H_2$ be the two components of $G\setminus \{u,v,w\}$.    
It is easy to see that $G$ is chordal (respectively, perfect) 
if and only if the subgraphs
$G_1$ and $G_2$ induced by $V(H_1)\cup\{u,v,w\}$ and  $V(H_2)\cup \{u,v,w\}$
are chordal (respectively, perfect).  We can recursively find edge separators
and separating triangles in the components till we are left with a collection
of induced subgraphs $G_1,G_2,\ldots G_k$ of $G$ such that none of them contains
an edge separator or a separating triangle.  That is, we have a decomposition of
$G$ into a collection of maximal W-near-triangulated subgraphs of $G$ 
such that, $G$ is chordal
(respectively, perfect) if and only if each of the subgraphs is chordal 
(respectively, perfect).  We call each maximal W-near-triangulated subgraph of 
$G$ a {\em W-component} of $G$.  This decomposition is a special case of the clique
decomposition described by \citet{Tarjan85}.  We need to show that the decomposition
can be done in linear time.  

The problem of finding edge separators in a $2$-connected
plane near-triangulation is reducible to finding separating triangles,
using a folklore algorithmic trick. 
Given a plane near-triangulation $G$ that is not already a triangulation,
we can triangulate $G$ by artificially adding a new vertex, say $p$, on
the external face of $G$ and making all vertices in the external face
of $G$ adjacent to $p$.  Let the new graph be denoted by $G_p$.  It is easy to see that any
edge separator $uv$ in $G$ must be a chord connecting two vertices in the
external face of $G$ and hence $puv$ must be a separating triangle in $G_p$.
Conversely, for any separating triangle $puv$ in $G_p$ containing the
newly added vertex $p$, $uv$ must be an edge separator in $G$. 

To construct $G_p$ from $G$ in linear time,
we need to find the vertices on the external face of 
$G$ from the adjacency list of $G$.  Here is one possible way to do this.  
We first embed $G$ in an $n \times n$ grid in linear time using the algorithm 
by \citet{schnyder1990embedding}. Now start from a the vertex, say $v_1$ in $G$ 
whose $x$ coordinate is the smallest. This vertex must be on the external
face of $G$. Traverse the adjacency list of $v_1$ to find the  
vertex $v_2$ such that the edge $v_1v_2$ has the largest slope among edges incident on $v_1$.
Clearly, $v_1v_2$ must be an edge on the external face of $G$.  
By traversing the adjacency list of $v_2$ once and finding the angles between $v_1v_2$ and $v_2w$
for each neighbour $w$ of $v_2$, we can identify the edge, say $v_2 v_3$ that appears next to $v_1 v_2$, in 
the clockwise ordering of edges around the vertex $v_2$. It is not difficult to see that the edge $v_2v_3$ is on the external face 
of $G$.
% Now think of the line from $v_1$ to $v_2$ as the new positive $x$ direction
% and find the vertex $v_3$ at 
% maximum angle with respect to the new ``positive x axis" from $v_2$.  
Continuing this way until we reach back $v_1$, we can find all the 
vertices on the boundary of $G$.
Since the adjacency list of each vertex is traversed at most once in the process, the procedure
takes only linear time. 
Adding the vertex $p$ to the adjacency list of every vertex on the boundary of the external face and adding the adjacency list of $p$ to $G$ can be done in linear time.
%Adding the vertex $p$ to the adjacency list of $G$
%and adding to the vertices on the boundary 
%is easily seen to require only linear time.  
   
Thus, starting from a $2$ connected plane near-triangulation $G$, we can 
construct a plane triangulation $G_p$ in linear time 
such that separating triangles in  $G_p$ correspond to either 
edge separators or separating triangles in $G$. 
It is well known that a plane triangulation is $4$-connected if and
only if it is free of separating triangles.  Thus, to find the maximal
W-components of $G_p$, it suffices to find the $4$-connected blocks
of $G_p$, which can be done in linear time using the algorithm by 
Kant\cite{kant1997more}.  It is not hard to see that, by removing
the vertex $p$ from each W-component of $G_p$ (whenever $p$ is
present), we can recover the W-components in $G$, which again requires
only linear time.  Hence we have:
\begin{lemma}\label{lem:wdecomposition}
Given a plane near-triangulation $G$, we can find the 
maximal W-near-triangulated subgraphs (W-components) of $G$ in linear time. 
Moreover, $G$ is chordal (respectively, perfect) if and only if 
each of the W-components is chordal (respectively, perfect).     
\end{lemma}
Consequently, we study W-near-triangulations for the rest of the paper.   
\section{Chordal plane near-triangulations}\label{sec2}
The following Lemma describes the structure of W-near-triangulations.
\begin{lemma}\label{lem:1}
 If $G$ is a W-near-triangulation with at least five vertices then for all  $u \in Int(G)$, $N[u]$ induces a wheel $W_k$ for some $k\ge5$.
\end{lemma}
\begin{proof}
 Let $G$ be a W-near-triangulation with at least five vertices and $u\in Int(G)$. As $u$ is an internal vertex, $|N(u)|\geq 3$. 
 Let $N(u)=\{u_0,u_1,u_2,\ldots u_{t-1}\}$ for some $t \ge 3$ such that $uu_0 , uu_1 , . . . , uu_{t-1}$ is the clockwise ordering of the edges incident with $u$.
 We claim that $u_i u_{i+1}\in E(G)$, where index $i \in \{0, 1, . . . , t-1\}$ is taken
 modulo $t$. Indeed, if $u_i u _{i+1}$ is not an edge, then $uu_i$ and $uu_{i+1}$ will be on the boundary of a face of length greater than three, 
 contradicting that $G$ is a W-near-triangulation. Consequently, $u_0,u_1,\ldots u_{t-1}u_0$ is a cycle and as $G$ is free of separating triangles, we get $t\geq 4$.  
 Now suppose that there exists an edge $u_i u_j$ with $j \notin \{i + 1, i-1\}$. Then $\{u, u_i , u_j \}$ will be a separating triangle. Therefore $N [u]$ is the wheel
 $W_{t+1}$, containing at least five vertices.  
\end{proof}
The next observation is directly verifiable and form the base case of the inductive argument that follows.
\begin{Observation} \label{obs:5}
Every plane near-triangulation with five or fewer vertices except $W_5$ is chordal. Every plane near-triangulation having no internal vertex is chordal.
\end{Observation}
%Since every cycle in a W-near-triangulation having no internal vertex is 
%a triangle, we have the following observation.  
%
%\begin{Observation}\label{Cor2}
% Every W-near-triangulation having no internal vertex is chordal.
% \end{Observation}
\begin{lemma}\label{lemma12}
A W-near-triangulation except $K_4$ is chordal iff it does not contain any internal vertices.
\end{lemma}
\begin{proof}
Let $G$ be a W-near-triangulation. If $|V(G)|\le 5$ then by Observation~\ref{obs:5},  $G$ is chordal iff $G$ is not $W_5$, which has an internal vertex. 
If $|V(G)|>5$ and there is no internal vertex in $G$, then by Observation~\ref{obs:5}, $G$  is chordal. If $|V(G)|>5$ and $G$ contains at least one 
internal vertex say $u$, then by the Lemma~\ref{lem:1}, $N[u]$ will induce a wheel say $W_k$ for $k\ge 5$ . 
As the rim of $W_k$ is a chordless cycle of length $(k-1)>3$, $G$ is not chordal.
%%Conversely, if $G$ is not chordal and $|V(G)|> 5$ then $G$ contains atleast one  chordless cycle $C_k$ for $k \ge 4$. 
%%Let $G'$ be the subgraph induced by $V(C_k)$ and $Int(C_k)$. As $G$ is near-triangulation, $G'$ should be near-triangulation. 
%%Since $C_k$ is chordless, we have $|Int(C_k)|\neq \phi$. That is, there exists at least one vertex in $Int(C_k)$, which must also be an internal vertex of $G$.
\end{proof}
The following theorem gives a local structural characterization for chordal plane near-triangulations in terms of  the closed neighbourhoods of internal vertices.
\begin{theorem}\label{thm:1}
 A plane near-triangulated graph is not chordal iff it contains an induced wheel of at least five vertices.
\end{theorem}
\begin{proof}
 Let $G$ be a plane near-triangulated graph. If $G$ contains an induced $W_k$ for some $k\ge5$ then the rim of $W_k$ is a chordless cycle of length exceeding three and hence $G$ is not chordal.
Conversely, if $G$ is not chordal, by Lemma~\ref{lem:wdecomposition}, we can decompose 
$G$ into its W-components - say $G_1,G_2,\ldots, G_t$ for some $t>0$ such that $G$ is not chordal if and only if at least one $G_i$, $1\leq i\leq t$ is not chordal. 
Let $G_i$ be a non-chordal W-component of $G$. Since $G_i$ is a plane near-triangulation which is not chordal, by Lemma~\ref{lemma12}, $G_i$ contains at least one internal vertex, say $v$.  
By Lemma~\ref{lem:1}, $N_{G_i}[v]$ induces a wheel $W_k$ for some $k\geq 5$ in $G_i$. Since $G_i$ is an induced subgraph of $G$ (by definition), $N_{G_i}[v]$ induces a 
wheel of at least five vertices in $G$ as well.
\end{proof}
% \begin{corollary}
%  A plane triangulated graph is not chordal iff it contains an induced wheel $W_k$ for some $k\ge5$.
% \end{corollary}
Lemma~\ref{lem:wdecomposition} and Lemma~\ref{lemma12}  yield a linear time algorithm for recognizing chordal plane near-triangulations, 
different from the standard method based on perfect elimination ordering \cite{RoseTL76}, as described below.
Given a plane near-triangulation $G$, it suffices to decompose $G$ into its W-components in linear time and check whether any of the components contain an internal vertex.  
Checking whether a plane near-triangulation contains an internal vertex requires only linear time (for instance, find vertices on the boundary as described in the previous section
and check whether the boundary includes every vertex or not).  Thus, in linear time, chordal plane 
near-triangulations can be recognized.  
\section{Perfect plane near-triangulations}\label{sec3}
Our next objective is to investigate the problem of providing a local characterization for plane near-triangulated perfect graphs similar
in spirit to Theorem~\ref{thm:1}. It is easy to see that the complement of cycle $C_n$ for $n\ge 7$ is not planar. 
Moreover, the complement of $C_5$ is isomorphic to $C_5$.  Thus, it follows from the strong perfect graph theorem \cite{chudnovsky2006strong} that to prove a plane triangulated graph $G$ is perfect, 
it is enough to prove that $G$ does not contain an induced odd hole.

Let $G$ be a plane near-triangulated graph. If $G$ contains an induced wheel on an
even number of vertices (even wheel) then clearly $G$ is not perfect. However the absence of an 
induced even wheel is not sufficient to guarantee the perfectness of a plane near-triangulation.
For example, the graph shown in Figure~\ref{fig1} does not contain any induced even wheel. But 
the vertices on the boundary of external face induce an odd hole. 

By Lemma~\ref{lem:wdecomposition}, we know that the problem
of characterizing perfect plane near-triangulations reduces to the problem
of characterizing perfect W-near-triangulations.  A 
local characterization that is simple enough to yield a {\em linear 
time} recognition procedure for arbitrary perfect W-near-triangulations 
appears hard to find.  Instead, 
we characterize a subclass of W-near-triangulations 
that indeed admits a simple local characterization that leads to a linear
time recognition procedure.  we derive a simple local structural characterization for 
W-near-triangulations that do not contain any induced wheel of five vertices.  
\begin{figure}[!htb]
\centering
 \includegraphics[scale=1]{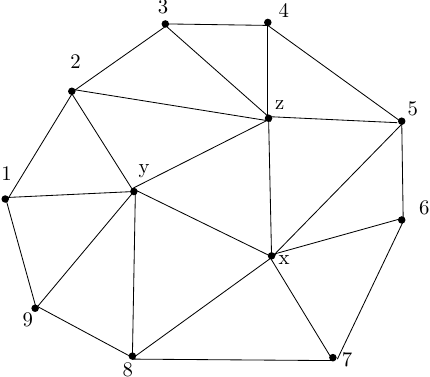}
\caption{An even wheel-free non-perfect plane near-triangulation}
\label{fig1}
\end{figure}
\section{$W_5$ free W-near-triangulations}\label{sec4}
In this section, we prove that any non-perfect $W_5$ free W-near-triangulation 
$G$ contains either an even wheel or contains three vertices forming an 
internal face such that the open neighbourhood of these vertices induces an odd hole.  Throughout this section, we use the notation $W(u)$ to denote a wheel with vertex $u$ at the centre. 
We first establish some properties of W-near-triangulations that will be useful for deriving the characterization.  
\begin{lemma}\label{lem:2}
 If a W-near-triangulation $G$ contains $e$ edges, $f$ internal faces and $t$ edges on the boundary of external face, 
then $f = t \bmod 2$. That is, $f$ is odd if and only if $t$ is odd.
\end{lemma}
\begin{proof}
 Each internal face is bounded with exactly three edges and each edge except those in the boundary of the external face is shared by two faces. 
 This implies $3f=2e-t$. Hence $t$ is odd if and only if $f$ is odd.
 \end{proof}
\begin{definition}[Face intersecting wheels]
 Let $W(u)$ and $W(v)$  ($u \ne v$) be any two wheels in a W-near-triangulation. $W(u)$ and $W(v)$ are said to be face intersecting if 
 they share  at least one face.
\end{definition}
\begin{lemma} \label{lemma2}
 Let $G$ be a W-near-triangulation and $W(u)$ and $W(v)$ ($u \ne v$) be any two face-intersecting 
 wheels in $G$,  then $W(u)$ and $W(v)$ share exactly two faces.  Further, the edge $uv$ 
 is on the boundary of these two faces. 
\end{lemma}
\begin{proof}
 Since $W(u)$ and $W(v)$ are face-intersecting and $u\ne v$, $u$ should be on the rim of $W(v)$. 
 Similarly $v$ should lie on the rim of $W(u)$. Hence the edge $uv$ should be a spoke in both the wheels. 
 As $u$ is on the rim of $W(v)$, $u$ will have exactly two neighbours (say $x$, $y$) on the rim of $W(v)$. 
 Similarly $v$ also have two neighbours (say $p$, $q$) on the rim of $W(u)$. If $p \ne x$ and $p \ne y$ then the edge $pu$ 
 will be a chord on the wheel $W(v)$ and the vertices $u,p,v$ forms a separating triangle in $G$, which is a contradiction to the definition of W-near-triangulation. Hence $p=x$ or $p=y$
 Similarly $q= y$ or $q = x$. This implies that either $p = x$ and $q=y$ or $p = y$ and $q=x$. 
 So $x$ and $y$ are the only vertices in $N(u) \cap N(v)$ and the edges $ux$ and $uy$ on the rim of $W(v)$ are also 
 spokes of $W(u)$ and $vx$ and $vy$ on the rim of $W(u)$ are
also spokes of $W(v)$. That is, $\{ux,xv,vu)\}$ and $\{uy,yv,vu\}$ are the only two faces shared by 
$W(u)$ and $W(v)$.
\end{proof}
\begin{corollary}\label{cor1}
Let $W(x)$, $W(y)$ and $W(z)$ (with $x\ne y\ne z$) be three (pair-wise) 
face-intersecting odd wheels in a W-near-triangulation $G$. Then they 
share exactly the common face $\{xy,yz,xz\}$. 
\end{corollary}
\begin{proof}
Since $W(x)$ and $W(y)$ are face-intersecting, by Lemma ~\ref{lemma2}, they share two faces 
(faces which has the edge $xy$ as one of its boundary). Similarly $W(y)$ and $W(z)$ share two faces 
(faces which has the edge $yz$ as one of its boundary) 
and $W(x)$ and $W(z)$  share two faces (faces which has the edge $xz$ as one of its boundary). 
 This implies that $xy$, $yz$ and $xz$ forms either a separating triangle or a face which is shared by $W(x)$, $W(y)$ and $W(z)$. But as $G$ is a W-near-triangulation, 
 the edges $xy$, $yz$ and $xz$ can not form a separating triangle.
\end{proof}
\begin{definition}[$W_{\Delta}$]
 Let $G$ be a W-near-triangulation and $W(x),W(y)$ and $W(z)$ be three face intersecting odd wheels in $G$. 
 If $N[x] \cup N[y] \cup N[z] \setminus \{x,y,z\}$ induces an odd hole in $G$, then the subgraph induced by  
 $N[x] \cup N[y] \cup N[z]$ is called a $W_{\Delta}$. The graph shown in Figure \ref{fig1} is an example of $W_{\Delta}$.
\end{definition}
If a W-near-triangulation $G$ with at least five vertices 
contains an internal vertex $u$ of odd degree
exceeding $3$, then
the rim of $W(u)$ induces an odd hole in $G$ and thus $G$ cannot be
perfect.  Since an internal vertex of degree $3$ would induce a separating
triangle, a W-near-triangulation with at least $5$ vertices cannot contain an internal
vertex whose degree is $3$.  
Consequently, the non-trivial case to handle is to characterize 
perfect W-near-triangulations whose
internal vertices are all of even degree (even W-near-triangulations).  
\begin{lemma}\label{lem5}
 Let $G$ be a $W_5$ free even W-near-triangulation and $W(x),W(y)$ and $W(z)$ be three face intersecting wheels in $G$. 
 Then $N[x] \cup N[y] \cup N[z]$ induces a $W_{\Delta}$.
\end{lemma}
 \begin{proof}
  Let $x$, $y$ and $z$ be three vertices of $G$ such that $W(x),W(y)$ and $W(z)$ are face intersecting wheels in $G$. 
  Let $G_1$ be the subgraph of $G$ induced by the vertices $x$, $y$, $z$ and their neighbours. 
  That is, $G_1$ is a subgraph of $G$ induced by $N[x] \cup N[y] \cup N[z]$. 
  Let $G_2$ be the subgraph of $G_1$ induced by $V(G_1)\setminus \{x,y,z\}$.
  If $G_1$ does not induce a $W_{\Delta}$ then there exists at least one chord in $G_2$. 
  Without loss of generality we may assume that there exists two non consecutive vertices $p$ and $q$ on the rim of wheels $W(x)$ and $W(y)$ respectively such that 
  $pq$ is a chord in $G_2$. We may further assume without loss of generality that
  there is no chord between the vertices of the clockwise boundary of $G_1$ from $p$ to $q$ (see Figure~\ref{Eg1}).
  % there does not exist any chord joining any pair of vertices between $p$ and $q$ in $G_1$. 

  Let $P=pq_1q_2\ldots q_r q$ (where $r \ge 1$) be the path joining $p$ and $q$ in $G_1$ (see Figure~\ref{Eg1}). 
  As $W(x)$ and $W(y)$ are face intersecting,there must be at least one vertex, say $q_i$, $1\leq i\leq r$) in $P$ that lies on the rim of both the wheels $W(x)$ and $W(y)$ (see Figure~\ref{Eg1}). 
 
 Let $s$ be the neighbour of $p$ on the rim of $W(x)$ in the anti clockwise direction and $t$ be the neighbour 
  of $q$ on the rim of $W_y$ in the clockwise direction (see Figure~\ref{Eg1}).
  Let $G_3$ be the subgraph of $G_1$ induced by the vertices $p,q_1,\ldots, q_i, \ldots, q_r,q$ and their neighbours except $x,y,s$ and $t$. 
  That is, $G_3$ is the subgraph of $G_1$ induced by the vertices
  $(N[p]\cup N[q_1] \cup..N[q_i]\cup N[q])\setminus \{x,y,s,t\}$ (see figure \ref{Eg2}). 
   \begin{figure}[h]
\centering
\begin{tikzpicture} [every node/.style={draw,circle,inner sep=2pt,minimum size=2pt,font=\footnotesize},
  node distance=1.75cm,on grid]
 \tikzstyle{peers}=[draw,circle,black,bottom color=white,top color= white, text=white,minimum width=1pt]
 
 \node (one) at (-6,3) {x};
 \node (two) at (-3.7,2.2) {y};
 \node (three) at (-7.8,2.8) {s};
 \node (four) at (-7.4,4) {p};
 \node (thirteen) at (-6,4.5) {$q_1$};
 \node (five) at (-4.5,4) {$q_i$};
 \node (fourteen) at (-3.0,4) {$q_r$};
 \node (six) at (-5,1) {z};
 \node (seven) at (-7,1.5) {};
 \node (eight) at (-6.5,0) {};
 \node (nine) at (-4.4,-0.4) {};
 \node (ten) at (-2,3) {q};
 \node (eleven) at (-2,2) {t};
 \node (twelve) at (-3,0) {};
  \draw (one) -- (three);
 \draw (fourteen) to (ten);
 \draw (one) to (thirteen);
 \draw [red] (one) --(four);
 \draw [red] (one) -- (five);

 \draw (one) -- (seven);
 \draw (one) -- (two);
 \draw (one) -- (six);
\draw (seven) -- (three);
\draw [red] (four) -- (three);
\draw (four) -- (thirteen);
\draw (seven) -- (six);
\draw (six) -- (two);
\draw [red] (five) -- (two);
\draw (two) -- (fourteen);
\draw [red] (two) -- (ten);
\draw (two) -- (eleven);
\draw [red] (ten) -- (eleven);
\draw (six) -- (eight);
\draw (six) -- (nine);
\draw (seven) -- (eight);
\draw (six) -- (twelve);
\draw (twelve) -- (eleven);
\draw (two) -- (twelve);
\draw (nine) -- (twelve);
 \draw [thick,dotted,bend left](thirteen) to (five);
  \draw [thick,dotted,bend right](fourteen) to (five);
   \draw [thick,dotted,bend left](nine) to (eight);
  \draw (four) .. controls (-7.8,5.2) and (-2.5,6) ..(ten);
 \draw [thick,dotted] (-3.9,3) -- (-3.5,3);
  \draw [thick,dotted,bend right] (-5.8,3.6) -- (-5.5,3.5);
  \draw [thick,dotted,bend right] (-5.5,0.5) -- (-5.0,0.3);
\end{tikzpicture}
\caption{$G_1$}
\label{Eg1}
\end{figure}
\begin{figure}
 \centering
\begin{tikzpicture} [every node/.style={draw,circle,inner sep=2pt,minimum size=2pt,font=\footnotesize},
  node distance=1.75cm,on grid]
 \tikzstyle{peers}=[draw,circle,black,bottom color=white,top color= white, text=white,minimum width=1pt]
 
 \node (one) at (0,-0.8) {$q_i$};
 \node (four) at (-3,-0.8) {$q_1$};
 \node (five) at (3,-0.8) {$q_r$};
 \node (two) at (5,-0.8) {q};
 \node (three) at (-5,-0.8){p};
 \draw  (three) to (four);
 \draw (five) to (two);
 \draw [thick,dotted] (one) to (four);
 \draw [thick,dotted] (one) to (five);
 \draw (three) .. controls (-5,3) and (5,3) ..(two);
 \draw (one) -- (-1,1);
 \draw (one) -- (1,1);
 \draw (three) -- (-2.7,0.8);
 \draw (three) -- (-3.7,1);
 \draw (two) -- (2.7,0.8);
 \draw (two) -- (3.7,1);
 \draw [thick,dotted] (-0.5,0.6) to (0.5,0.6); 
 \draw [thick,dotted] (-4,-0.1) to (-4.3,0.2); 
 \draw [thick,dotted] (3.9,0) to (4.3,0.1); 
%  \draw (two) to (6.5,-1.5);
\end{tikzpicture}
\caption{$G_3$}
\label{Eg2}
\end{figure}
  Let  $e$, $f$, $n_e$ and $n_i$ be the number of edges, internal faces, external vertices and internal vertices in $G_3$ 
  respectively. Let $n=n_e+n_i$ be the total number of vertices in $G_3$.
 Since $G_3$ is internally triangulated, by Lemma~\ref{lem:2} we have,
  \begin{equation}  \label{eq1}
  3f=2e-n_e 
\end{equation}    
Using Euler's formula \cite{west1996introduction} we get:
   \begin{equation}\label{eq2}
    3n_e+3n_i+3f = 3e+3
   \end{equation}
 from $\eqref{eq1}$ and $\eqref{eq2}$ we get:
\begin{equation}\label{eq3}
 e=2n_e+3n_i-3
\end{equation}
Let $V_i$ and $V_e$ be the set of internal and external vertices in $G_3$ respectively.
As $G_3$ is an induced subgraph of $G_1$ and $pq$ a chord in $G_1$, all vertices except $p$ and $q$ in $V_e$ are 
internal vertices of $G_1$ (see Figure~\ref{Eg1} and Figure~\ref{Eg2}). 
Also every vertex in $V_e$ except $q_i$ must either be on the rim of $W_x$ or on the rim of $W_y$, but not on both. 
That is, for all external vertices  in $G_3$ except $p$, $q$ and $q_i$, all but one of their neighbours in $G_1$ must be  
in the graph $G_3$. It follows that the degree of all vertices on the external face of $G_3$ except $p$, $q$ and $q_i$ 
will be at least five. This is true because we have assumed that $G_1$ is a $W_5$ free even W-near-triangulation and 
hence has no internal vertex of degree below six. 

The vertices $p$, $q$ and $q_i$ have two neighbours on the cycle $p,q_1,\ldots, q_i, \ldots, q_r,q$ and 
as $G_3$ is plane near-triangulated, they must have at least one neighbour in $V_i$. So the
degree of $p$, $q$ and $q_i$ will be at least three in $G_3$. 
Since every neighbour (in $G_1$) of vertices in $V_i$ is also present in $G_3$, 
degree of all vertices in $V_i$ must be at least six in $G_3$. Counting the degree of vertices, we get:
\begin{equation}
  2e \ge 6n_i+5(n_e-3)+ 9
\end{equation}

Substituting $\eqref{eq3}$ we get,
\begin{equation}
 4n_e+6n_i-6 \ge 6n_i+5n_e-6 \Longrightarrow 0 \ge 2n_e \Longrightarrow 0 \ge n_e
\end{equation}
which is a contradiction.

\end{proof}

The following lemma shows that Lemma~\ref{lem5} characterizes all non-perfect $W_5$ free even W-near-triangulations. 

\begin{lemma}\label{lem6}
 Every $W_5$ free even W-near-triangulation $G$ that contains an induced odd hole must contain an induced $W_{\Delta}$.
 \end{lemma}

\begin{proof}
 Let $C$ be an induced odd hole in $G$.  As $G$ is a plane near-triangulation, there must exist at least one vertex in $Int(C)$. 
 Let $G'=(V',E')$ be the subgraph induced by the vertices in $Int(C)$. 
 If $|V'|=1$ then $V' \cup V(C)$ will have to induce an odd wheel which is impossible as $G$ is an even W-near-triangulation.   
 The case $|V'|=2 $ is also not possible as two face intersecting odd wheels will not induce an odd hole (See Lemma ~\ref{lem:2}).
 Thus we may assume that  $|V'|\ge 3$.  Let $G''$ be the subgraph of $G$ induced 
  by the vertices $V' \cup V(C)$.  We have to consider the following cases. 
\begin{enumerate} 

  \item $G'$ contains an induced triangle $\triangle=(xyz)$:  In this case, $W(x)$, $W(y)$ and $W(z)$ are face intersecting odd wheels and  
  by Lemma~\ref{lem5}, $N(x)\cup N(Y) \cup N(z)$ induces an odd hole, proving the lemma.  
 
  \item $G'$ does not contain any induced triangles: In this case, as $G$ is a plane near-triangulation, the only possibility is that $V'$
induces a tree $T$ of at least $3$ vertices. (See Figure ~\ref{Eg4a}.  $T$ could possibly a path as in Figure~\ref{Eg4}.) 
Let $v_0,v_1,\ldots, v_r$  for some $r\ge 2$ be the vertices in $T$ ordered in such a way that $v_0$ is the root of the tree and each node $v_j$ for $j>0$ is a child of some unique $v_i$, $i<j$ in $T$.  
 Note that every neighbour of 
 a vertex $v_i$ in $T$ except its children and its parent in the tree $T$
 must be a vertex in the odd hole $C$.  
 Let $W(v_i)$  be the wheel induced by $N[v_i]$. Since $G$ is an even $W_5$ free plane near-triangulation, $v_i$ must have even degree 
  (greater than 4) for each $i \in \{0,1,\dots, r\}$. Further, each edge incident on $v_i$ ($0\le i \le r$) is 
  shared by exactly two internal faces in $G''$ (See Figure ~\ref{Eg4a}).  
Hence, for each  $i,j \in \{0,1,\dots, r\}$, if $v_i$ is the parent
of $v_j$ in the tree $T$, the wheels $W(v_i)$ and $W(v_j)$ must be face intersecting 
odd wheels sharing exactly two faces. Using this observation, we count the 
the total number of internal faces in $G''$ to be  
  $f= \sum_{i=0}^{r} deg(v_i) - 2(r-1)$.  As the degree of every internal vertex in $G''$ is even,  $f$ must be even. Then by Lemma~\ref{lem:2}, the number of 
external vertices of $G''$ should be even. That is,  $|V(C)|$ must be even.  However, this contradicts the assumption that $C$ is an odd hole. 
   \end{enumerate}
% 
%  \item[b)] $V'$ induces a path say $v_0,v_1,\ldots, v_r$ for some $r\ge 2$ with the neighbours of every vertex $v_i$, $i \in \{0,1,\dots, r\}$ inducing 
%  an even wheel, which we denote as $W(v_i)$.  Note that all the neighbours of $v_i$ (except $v_{i-1}$ and $v_{i+1}$ for $0<i<r$) must be
%  in $C$, for otherwise, $Int(C)$ will contain a triangle.    Moreover, since $G$ is $W_5$ free, $v_i$ must have even degree (greater than 4) 
%  for each $i \in \{0,1,\dots, r\}$. Further, each edge incident on $v_i$ ($0\le i \le r$) is shared by exactly two
%  internal faces in $G''$ (See Figure ~\ref{Eg4}).  Hence $\forall i \in \{1,2,\dots, r\}$, the wheels $W(v_i)$ and 
%  $W(v_{i-1})$ must be face intersecting even wheels. Consequently, the total number of internal faces in $G''$,  
%  $f= \sum_{i=0}^{r} deg(v_i) - 2(r-1)$.  As the degree of every internal vertex in $G''$ is even,  $f$ must be even. Then by Lemma~\ref{lem:2} number of 
%  external vertices should be even. That is,  $|V(C)|$ must be even.  However, this contradicts the assumption that $C$ is an odd hole. 
% \end{itemize}
   
   \begin{subfigures}
 
   \begin{figure}[h]
\centering
\begin{tikzpicture} [every node/.style={draw,circle,inner sep=2pt,minimum size=2pt,font=\footnotesize},
  node distance=1.75cm,on grid]
 \tikzstyle{peers}=[draw,circle,black,bottom color=white,top color= white, text=white,minimum width=1pt]
 
 \node (one) at (-6,2.5) {$v_0$};
 \node (two) at (-3.7,2.2) {$v_1$};
 \node (twenty) at (-2,2.5) {$v_3$};
  \node (twenty one) at (-2.5,1) {$v_4$};
 \node (three) at (-7.8,2.8) {};
 \node (four) at (-7.4,4) {};
 \node (thirteen) at (-6,4.5) {};
 \node (five) at (-4.5,4) {};
 \node (fourteen) at (-3.0,4) {};
\node (six) at (1,2) {};
 \node (seven) at (-7,0.5) {};
\node (eight) at (0,3) {};
%  \node (nine) at (-4.4,-0.4) {};
 \node (ten) at (-2,4) {};
 \node (eleven) at (-5,1) {$v_2$};
  \node (fifteen) at (0,2) {$v_r$};
 \node (twelve) at (-4,0) {};
 \node (sixteen) at (1,1) {};
 \node (seventeen) at (0,0) {};
  \draw (fifteen) -- (sixteen);
  \draw (six) -- (sixteen);
  \draw [thick,dotted,bend left] (sixteen) to (seventeen);
  \draw (fifteen) -- (seventeen);
  \draw (one) -- (three);
\draw (fourteen) to (ten);
 \draw (one) to (thirteen);
 \draw (one) --(four);
 \draw (one) -- (five);
 \draw (one) -- (seven);
\draw (one) -- (eleven);
\draw (twenty) -- (ten);
 \draw (one) -- (two);
 \draw (fifteen) -- (six);
\draw (seven) -- (three);
\draw (four) -- (three);
\draw (four) -- (thirteen);
 \draw [thick,dotted,bend left](ten) to (eight);
\draw (six) -- (eight);
\draw (five) -- (two);
\draw (two) -- (fourteen);
\draw (two) -- (ten);
\draw (seven) -- (eleven);
\draw (two) -- (twenty);
\draw [bend right](twelve) to (one);
\draw [thick,dotted,bend right] (seven) -- (twelve);
\draw (fifteen) -- (eight);
\draw [thick,dotted,bend left](seventeen) to (twelve);
\draw (twelve) -- (eleven);
\draw (two) -- (twelve);
 \draw (twenty one) -- (two);
  \draw (twenty one) -- (twelve);
  \draw (twenty one) -- (seventeen);
  \draw [bend left](two) to (seventeen);
 \draw [thick,dotted,bend left](thirteen) to (five);
  \draw [thick,dotted,bend right](fourteen) to (five);
%    \draw [thick,dotted,bend left](nine) to (eight);
%   \draw (four) .. controls (-7.8,5.2) and (-2.5,6) ..(ten);
 \draw [thick,dotted] (-3.9,3) -- (-3.5,3);
  \draw [thick,dotted] (twenty) -- (fifteen);
  \draw [thick,dotted,bend right] (-5.8,3.6) -- (-5.5,3.5);
%   \draw [thick,dotted,bend right] (-5.5,0.5) -- (-5.0,0.3);
\end{tikzpicture}
\caption{$V'$ induces a tree}
\label{Eg4a}
%\caption{\label{first}Caption text.}
\end{figure}
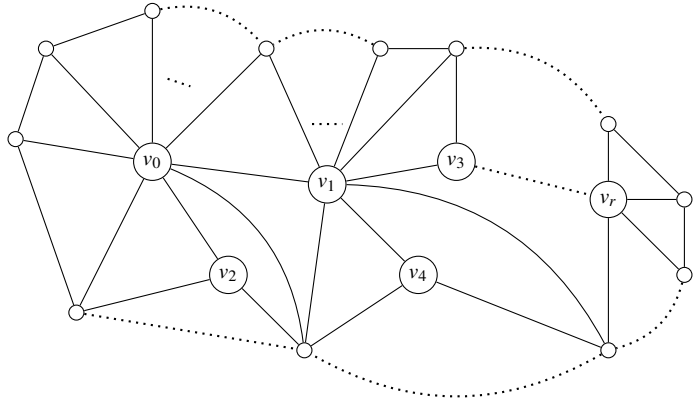
   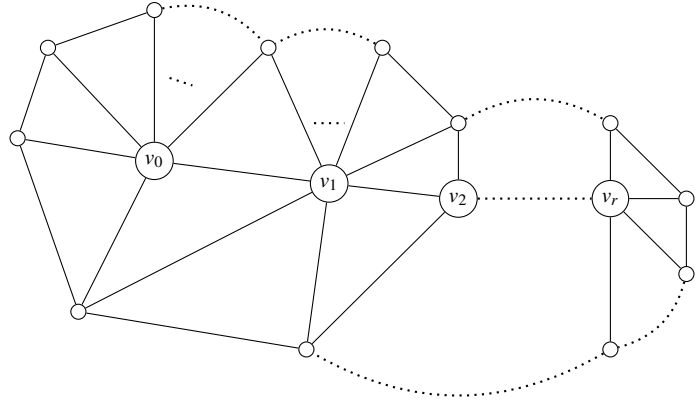
\begin{figure}[h]
\centering
\begin{tikzpicture} [every node/.style={draw,circle,inner sep=2pt,minimum size=2pt,font=\footnotesize},
  node distance=1.75cm,on grid]
 \tikzstyle{peers}=[draw,circle,black,bottom color=white,top color= white, text=white,minimum width=1pt]
 
 \node (one) at (-6,2.5) {$v_0$};
 \node (two) at (-3.7,2.2) {$v_1$};
 \node (three) at (-7.8,2.8) {};
 \node (four) at (-7.4,4) {};
 \node (thirteen) at (-6,4.5) {};
 \node (five) at (-4.5,4) {};
 \node (fourteen) at (-3.0,4) {};
\node (six) at (1,2) {};
 \node (seven) at (-7,0.5) {};
\node (eight) at (0,3) {};
%  \node (nine) at (-4.4,-0.4) {};
 \node (ten) at (-2,3) {};
 \node (eleven) at (-2,2) {$v_2$};
  \node (fifteen) at (0,2) {$v_r$};
 \node (twelve) at (-4,0) {};
 \node (sixteen) at (1,1) {};
 \node (seventeen) at (0,0) {};
  \draw (fifteen) -- (sixteen);
  \draw (six) -- (sixteen);
  \draw [thick,dotted,bend left] (sixteen) to (seventeen);
  \draw (fifteen) -- (seventeen);
  \draw (one) -- (three);
\draw (fourteen) to (ten);
 \draw (one) to (thirteen);
 \draw (one) --(four);
 \draw (one) -- (five);
 \draw (one) -- (seven);
 \draw (two) -- (seven);
\draw (eleven) [thick,dotted] to (fifteen);
 \draw (one) -- (two);
 \draw (fifteen) -- (six);
\draw (seven) -- (three);
\draw (four) -- (three);
\draw (four) -- (thirteen);
 \draw [thick,dotted,bend left](ten) to (eight);
\draw (six) -- (eight);
\draw (five) -- (two);
\draw (two) -- (fourteen);
\draw (two) -- (ten);
\draw (two) -- (eleven);
\draw (ten) -- (eleven);
% \draw (six) -- (eight);
\draw (seven) -- (twelve);
\draw (fifteen) -- (eight);
\draw [thick,dotted,bend left](seventeen) to (twelve);
\draw (twelve) -- (eleven);
\draw (two) -- (twelve);
% \draw (nine) -- (twelve);
 \draw [thick,dotted,bend left](thirteen) to (five);
  \draw [thick,dotted,bend right](fourteen) to (five);
%    \draw [thick,dotted,bend left](nine) to (eight);
%   \draw (four) .. controls (-7.8,5.2) and (-2.5,6) ..(ten);
 \draw [thick,dotted] (-3.9,3) -- (-3.5,3);
  \draw [thick,dotted,bend right] (-5.8,3.6) -- (-5.5,3.5);
%   \draw [thick,dotted,bend right] (-5.5,0.5) -- (-5.0,0.3);
\end{tikzpicture}
\caption{$V'$ induces a path}
\label{Eg4}
%\caption{\label{first}Caption text.}
\end{figure}
   
   \end{subfigures}
%   For $\forall i \in \{1,2,\ldots r-1\}$, 
%    $v_i-1$ and $v_i+1$ lies on the rim of $W_{v_i}$ and they are the only neighbours of $v_i$ that 
%    exists in $Int(C)$. That is all neighbours of $v_i$ except $v_{i-1}$ and $v_{i+1}$ exists on the cycle $C$. 
%    And $N(v_i) \cap N(v_{i-1}) = 2$. 
\end{proof}
The proof of Lemma~\ref{lem6} shows that if a $W_5$ free even W-near-triangulation $G$   
contains an odd hole, then the interior of the odd hole cannot be a tree, and 
hence must contain a facial triangle $uvw$. On the 
other hand, if three internal vertices $u,v$ and $w$ forms a facial triangle in a $W_5$ free even W-near-triangulation $G$,
by Lemma~\ref{lem5}, the neighbours of the facial triangle $uvw$ must induce an odd hole. 
Hence, we have the following computationally useful corollary.  
\begin{corollary}\label{cor:computation}
 Let $G$ be a $W_5$ free even W-near-triangulation.  The following conditions
 are equivalent.  
 \begin{enumerate}
  \item $G$ is not perfect.
  \item The subgraph induced by vertices of $Int(G)$ contains a facial triangle. 
  \item The subgraph induced by vertices of $Int(G)$ is not a tree.  
 \end{enumerate}
\end{corollary}
Lemma~\ref{lem5} and Lemma~\ref{lem6} yields the following local characterization
for perfect W-near-triangulations.  
\begin{theorem}\label{thm:3}
 A $W_5$ free plane triangulated W-near-triangulation $G$ other than a $K_4$ 
 is perfect if and only if the following conditions hold
 \begin{itemize}
  \item $G$ does not contain an even wheel
  \item $G$ does not contain an induced $W_{\Delta}$
 \end{itemize}
\end{theorem}
%\begin{proof}
% Let $G(V, E)$ be a $W_5$ free plane W-triangulated graph. If $G$ contains an induced even wheel or an induced $W_{\Delta}$, 
% then $G$ contains an induced odd hole and hence $G$ is not perfect. 
%
%Conversely, let $G$ is not perfect. Decompose $G$ in to W-components using the method discussed in Section~\ref{sec1}. $G$ is not perfect iff one of the $w$ components is not perfect. As $G$ is $W_5$ free, $w$ components are also $W_5$ free.A $W_5$ free  $w$near-triangulation is not perfect iff it contains an even wheel or an induced $W_{\Delta}$ (See Lemma~\ref{lem6}).
% Otherwise, let $\triangle = 
% (u,v,w)$ be a separating triangle in $G$. Then by the Observation~\ref{obs}, $(u,v,w)$ divides $G$ 
% into two components say $G_1$ and $G_2$ and either $G_1$ or $G_2$ or both are not perfect. 
% The claim follows by induction.
%\end{proof}
A simple linear time algorithm for checking whether a given $W_5$ free 
W-near-triangulation $G$ is perfect follows from Corollary~\ref{cor:computation} and Theorem~\ref{thm:3}, as explained below. We can find the vertices on the external face of $G$ in linear time using the method described in Section~\ref{wdecomposition} and create an array whose $i^{th}$ entry indicates whether the $i^{th}$ vertex is in $Int(G)$ or 
not, in linear time.  Now we can check whether any vertex in $Int(G)$ has 
odd degree, in which case, we immediately conclude that $G$ is not perfect.  
Otherwise, we perform a breadth first search on the subgraph induced 
by the vertices of $Int(G)$  (the indicator array serves to ensures that the search 
never enters a vertex on the boundary of $G$) in linear time to decide whether 
the subgraph induced by vertices of $Int(G)$ is a tree (Corollary~\ref{cor:computation}).  
% \begin{Note}
% ~
% \begin{enumerate}
%  \item  
%  \item Given an even W triangulation $G$ that is not perfect.  It might happen that the neighbours of none of the edges or facial
%  triangles in $G$ induces an odd hole, as shown in Figure~\ref{TO ADD}.    This check results in quadratic time complexity of the algorithm described in the next section.  
% \end{enumerate}
% \end{Note}
\section{Discussion and Conclusion}
Investigation into the structure of perfect plane-triangulations or plane 
near-triangulations has been reported in two unpublished manuscripts in the literature.  
In a work done prior to this paper by~\citet{benchetrit2015},
a structural characterization for perfect plane triangulations
that is not a local characterization is reported.  
The characterization does not appear to yield any direct algorithmic consequences.  
In a work  done subsequent to this paper by~\citet{sameera.arxiv}, 
a local characterization for W-near-triangulations has been derived
using a different proof technique, but the characterization is more 
complex than the one derived here for $W_5$ free W-near-triangulations and 
consequently yields only a quadratic time recognition algorithm for perfectness.  
It is interesting to check whether the approach presented here can be extended to 
all W-near-triangulations in a way to yield a linear time recognition algorithm 
for plane perfect near-triangulations.  
\section{Acknowledgment}
We thank Ajit A. Diwan, IIT Bombay and Jasine Babu, IIT Palakkad for discussions and suggestions. A preliminary version of this paper was presented in CALDAM 2019.
\bibliographystyle{plainnat}
\section*{References}
\bibliography{simple}
\end{document}